\documentclass[a4paper,10pt,reqno]{amsart}

\usepackage{amsmath}
\usepackage{amssymb}
\usepackage{amsthm}
\usepackage{amsaddr}
\usepackage{mathrsfs}
\usepackage{enumerate}

\numberwithin{equation}{section}  

\newtheorem{theorem}{Theorem}[section]
\newtheorem{lemma}[theorem]{Lemma}
\newtheorem{proposition}[theorem]{Proposition}
\newtheorem{remark}[theorem]{Remark}

\newcommand{\proba}{{\mathbb{P}}} 
\newcommand{\zentiers}{\mathbb{Z}} \newcommand{\entiers}{\mathbb{N}}

\newcommand{\alphabet}{\mathcal{A}} \newcommand{\tq}{\;\big|\;}
 
 \newcommand{\paire}{\mathfrak{Y}}

 \newcommand{\somb}{\Gamma}

\begin{document}
\title[Blurred chains.]{Blurred stochastic chains.}

\author{Pierre Collet }
\address{
  Centre de Physique Th\'eorique, CNRS UMR 7644, Ecole Polytechnique,
  91128 Palaiseau Cedex, France} 
 \email{collet@cpht.polytechnique.fr}

\author{Antonio Galves} 
\address{
  Instituto de Matem\'atica e Estat\'{\i}stica, Universidade de S\~ao
  Paulo, BP 66281, 05315-970 S\~ao Paulo, Brasil} 
  \email{galves@ime.usp.br}

\date{November 1, 2016}

  \thanks{This work is part of FAPESP project
    {\em Research, Innovation and Dissemination Center for
      Neuromathematics} (grant 2013/07699-0) and USP project {\em
      Mathematics, computation, language and the brain}.  AG is
    partially supported by CNPq fellowship (grant 309501/2011-3.)}

\keywords{stochastic chains, coupling, random perturbations}  
\subjclass[2000]{60G99, 60K99}

\begin{abstract}
  Assume we have two stochastic chains taking values in a finite
  alphabet.  These chains may be of infinite order. Assume also that
  these chains are coupled in such a way that given the past
  of both chains
  they have a not too large probability of differing.  This is the case when
  we observe a chain through a noisy channel. This situation
  presumably also occurs in models for the brain activity when a chain
  of stimuli is presented to a volunteer and we observe a
  corresponding chain of neurophysiological recordings.

  The question is how these two chains are quantitatively related.
  Under suitable conditions, we obtain upper-bounds for the
  differences between the marginal conditional distributions of the
  two chains and between the probability of the next symbol of each
  chain, given the past of the past of one of them. 
\end{abstract}

\maketitle

\section{Introduction.}

Assume $(X_n)_{n\in\zentiers}$ and $(Y_n)_{n\in\zentiers}$ are
stochastic chains coupled in such a way that given the past they have
a small probability of differing.  The simplest situation is when
$(X_n)$ is an autonomous chain, possibly of infinite order and each
step $n$ the symbol $Y_n$ is obtained by changing with small
probability the symbol $X_n$ (Collet {\sl et al.} 2008 and Garcia and
Moreira 2015)\nocite{cgl} \nocite{garcia}. In this case, if $(X_n)$ is
not of infinite order but only a Markov chain, the pair $(X_n, Y_n)$
is an example of Hidden Markov Model (we refer the reader to the
classical references Baum and Petri 1966 and Rabiner 1989; see also
Verbitsky 2015 for a recent survey on the more general class of Hidden
Gibbs Models). However, besides the fact that articles on Hidden
Markov Models only consider Markov chains, the classical literature on
these models, as far as we know, do not consider the type of results
proved here.

A more involved situation appears in neurobiology when
electrophysiological or behavioral data are recorded while a volunteer
is exposed to a sequence of stimuli generated by a stochastic
chain. Experimental evidence support the idea that the value
associated to the recordings at each step is a
marker indicating how well the brain of the volunteer the predicts the next step of the
stimulus, given the past. In this situation the chains are coupled in
a more complicated way than just independent random
perturbations. More precisely, in this case the law at each step of
the recorded value may depend on the past of both chains (Duarte {\sl et
  al.} 2016).

A more complicated situation occurs when the next step of each chain
depends on the past of both chains. This situation occurs when we model
the joint behavior of two opponents trying to guess each other next response,
given their knowledge of the past. In this case each chain can be
seen as {\sl blurred} version of the other.

In what follows we present a mathematical framework covering this more
general case. In this framework we will make 
  assumptions on one of the chain (for definiteness the chain $ (X_n)$), and
derive some consequences for the other chain (for definiteness the 
chain $(Y_n)$). For example , we obtain
upper-bounds for the differences between the marginal conditional
distributions of 
$(X_n)$ and $(Y_n)$. This is the content of our Theorem
\ref{blurred}.  We also ask how well can we predict the next symbol of
the chain $(X_n)$, given that we know the symbols of the chain $(Y_n)$
up to the present time. This is the content of our Theorem \ref{predire}.

This article is organized as follows. The notation, basic definitions
and the main results (Theorems \ref{blurred} and \ref{predire}) are
stated in Section \ref{notation}.  The basic properties of the
marginal chains are presented in Section \ref{marginal}. These results
will be used in the proofs of the main results and are interesting by
themselves. The lemmas required in the proofs of  Theorems
\ref{blurred} and \ref{predire} are presented in Section
\ref{lemmas}. Finally the proofs of Propositions \ref{leX} and
\ref{leY} and Theorems \ref{blurred} and \ref{predire} are presented
in Section \ref{proofs}.

\section{Notation and main results.}\label{notation}

Let $\alphabet$ denote a finite alphabet.
 Given two integers $m\leq n$ we denote by $a_m^n$ the sequence
$a_m, \ldots, a_n$ of symbols in $\alphabet$.
The length of the sequence $a_m^n$
 is denoted by $\ell(a_m^n)$ and is given by $\ell(a_m^n) = n-m+1$.  
Any sequence $a_m^n$ with $m > n$ represents the empty
string. 
We will also use the notation $\eta_{a}^{b}=(x_{a}^{b},y_{a}^{b})$ for
a sequence 
$$
\eta_{j}=(x_{j},\,y_{j})\in\alphabet^{2}\qquad a\le j\le b\;.
$$  

Let $(\paire_{n})_{n\in\zentiers}=(X_{n},\,Y_{n})_{n\in\zentiers}$ 
be a stationary stochastic
chain taking values in $\alphabet^{2}$.

The {\sl blurring effect} is measured by the quantity
\begin{equation*}
\rho =\sup_{\substack{a\in \alphabet,\;k\ge 0\\
\eta_{-k}^{-1}\in (\alphabet^{2})^{k}}} \quad \sum_{b\neq a}
\proba\big(Y_{0}=b\tq X_{0}=a\,,\, \paire_{-k}^{-1}=\eta_{-k}^{-1}\big)\;.
\end{equation*}

Before presenting our main results, we need to introduce two hypotheses.

\textbf{Hypothesis H1} says that the blurring effect is smaller that 1

\begin{equation*}
\rho <1\;. \tag{H1}
\end{equation*}

\textbf{Hypothesis H2}  refers to the non-nullness the chains, namely
\begin{equation}
\alpha=
\inf_{\substack{a\in \alphabet,\, k\ge 1\\
\eta_{-k}^{-1}\in
    (\alphabet^{2})^{k}}}\proba(X_{0}=a\tq \paire_{-k}^{-1}= 
\eta_{-k}^{-1}\big)>0\;.\tag{H2}
\end{equation}

\begin{remark}
If the probability of discrepancy between the symbols $X_{0}$ and
$Y_{0}$ conditioned to the past satisfies for any $k\ge 0$
$$
\sup_{\eta_{-k}^{-1}\in(\alphabet^{2})^{k}}\proba\big(X_{0}\neq
Y_{0}\tq \paire_{-k}^{-1}=\eta_{-k}^{-1}\big)<{\alpha}\;,
$$
then hypothesis H1 holds. The proof is left to the reader.
\end{remark}

We will use the notations  (for $k\ge j\ge 1$)
%\marginpar{revoir cette notation, seulement $\Gamma_{k}$ ?}
\begin{equation}\label{exH3}
\somb_{j,\,k} = \sum_{\ell=1}^{j} \beta_{\ell,\,k} \;,% \tag{H3}
\end{equation}
where 
$$
\beta_{j,\,k}=\sup_{\substack{a\in\alphabet,\,x_{-j}^{-1}\in
    \alphabet^{j}\\
\eta_{-k}^{-j-1}\in(\alphabet^{2})^{k-j}\\
\tilde   \eta_{-k}^{-j-1}\in(\alphabet^{2})^{k-j}}}
\log\left(\frac{\proba(X_{0}=a\tq X_{-j}^{-1}=x_{-j}^{-1},\;
\paire_{-k}^{-j-1}= \eta_{-k}^{-j-1})}{
\proba(X_{0}=a\tq X_{-j}^{-1}=x_{-j}^{-1},\;
\paire_{-k}^{-j-1}= \tilde\eta_{-k}^{-j-1})}\right)\;.
$$

In our previous work (Collet {\sl et al.} 2008) we assumed
  (among other things) that  the
chains were of infinite order and satisfied continuity, namely
$\Gamma_{\infty,\,\infty}<\infty$. In the present work we do not require
these assumptions.

We may now state our main results. 
It will be convenient in order to alleviate the notation to define a
positive 
function $R$ on $(0,1]\times \entiers\times[0,1)$ by
$$
R(\alpha,\,k,\,\rho)=2+
 \frac{e^{2\,\somb_{k,\,k}}\;
\left[2\,\big(e^{\somb_{k,\,k}}-1\big)+\big(e
^{2\,\somb_{k,\,k}}-1\big)\right]}{\alpha\;(1-\rho)^{2}}
+2\;e^{\somb_{k,\,k}}\;\big(e^{\somb_{k,\,k}}-1\big)\;.
$$

\begin{theorem}\label{blurred} Assume that Hypotheses H1 and H2  hold.
Then  for any $j\geq0$,
$$
\sup_{a\in \alphabet, w_{-j}^{-1}\in \alphabet^{j}}\quad
\bigg|
\proba\big(Y_{0}=a\tq Y_{-j}^{-1}=w_{-j}^{-1}\big)-
\proba\big(X_{0}=a\tq X_{-j}^{-1}=w_{-j}^{-1}\big)\bigg|
\le\rho\;R(\alpha,\,j,\,\rho)\;.
$$
Moreover, for any $a\in \alphabet$, any integer $j$, any
$w_{-j}^{-1}\in \alphabet^{j}$, if  
$\rho\;R(\alpha,\,j,\,\rho)<\alpha$
we have
$$
1- \rho\;\frac{R(\alpha,\,j,\,\rho)}{\alpha}
\le \frac{\proba\big(Y_{0}=a\tq Y_{-j}^{-1}=w_{-j}^{-1}\big)}
{\proba\big(X_{0}=a\tq X_{-j}^{-1}=w_{-j}^{-1}\big)}
\le 1+ \rho\;\frac{R(\alpha,\,j,\,\rho)}{ \alpha}
$$
\end{theorem}

\begin{theorem}\label{predire}
 Assume that Hypotheses H1 and H2 hold. Then for any integer $k\ge 0$, 
and for any
$\rho>0$ we have
$$
\sup_{a\in\alphabet \,,w_{-k}^{-1}\in\alphabet^{k}}
\big|\proba\big(Y_{0}=a\tq Y_{-k}^{-1}= w_{-k}^{-1}\big)-
\proba\big(X_{0}=a\tq Y_{-k}^{-1}=w_{-k}^{-1}\big)\big|\le \rho\;.
$$
If moreover 
$\rho\;R(\alpha,\,k,\,\rho)<\alpha $, we have for any 
 $a\in\alphabet$, and for any $y_{-k}^{-1}\in \alphabet^{k}$
$$
1-\frac{\rho}{\alpha-\rho\;
R(\alpha,\,k,\,\rho)}
\le
\frac{\proba\big(X_{0}=a\tq Y_{-k}^{-1}=y_{-k}^{-1},\,\big)}
{\proba\big(Y_{0}=a\tq Y_{-k}^{-1}=y_{-k}^{-1}\big)}\le
1+
\frac{\rho}{\alpha-\rho\;R(\alpha,\,k,\,\rho)}\;.
$$
\end{theorem}
The proofs will be given in Section \ref{proofs}.

\section{Properties of the marginal chains.}\label{marginal}

In this section we state some results about the two marginal chains $
(X_n) _{n\in\zentiers}$ and $(Y_n) _{n\in\zentiers}$ 
which follow from the Hypotheses H1, H2. These results will be useful latter.

\begin{proposition}\label{leX}
%\marginpar{demonstration anterieure incorrecte.}
Under the hypothesis H2  the process $X$ satisfies
\begin{enumerate}
\item {\em Non-nullness}, that is for any $k\ge0$
  \begin{equation*}
 \inf_{a\in
 \alphabet,\; x_{-k}^{-1}\in  
\alphabet^{k}}\; \proba\big(X_{0}=a\tq  X_{-k}^{-1}=x_{-k}^{-1}\big)\ge \alpha
  \end{equation*}
  \item For any $k\ge j\ge 1$ we have
$$
   \sup_{\substack{ v_{-k}^{-1}, \,x_{-k}^{-1} \in  
  \alphabet^{k},\\
a\in \alphabet,\,x_{-j}^{-1}= v_{-j}^{-1} }} 
\;\;\log\left(\frac{\proba\big(X_{0}=a\tq
  X_{-k}^{-1}=x_{-k}^{-1}\big)}
{\proba\big(X_{0}=a\tq  X_{-k}^{-1}=v_{-k}^{-1}\big)}\right) 
\le 2\,\beta_{j,\,k}.
$$ 
\end{enumerate}
\end{proposition}
The proof will be given in Section \ref{proofs}.

\begin{proposition}\label{leY}
 Assume hypothesis H1 and H2 hold. Then 
for any $a\in\alphabet$, for any integers $k>j\ge 0$, for any 
$y_{-j}^{-1}\in \alphabet^{j}$, for any 
$y_{-k}^{-j-1}\in \alphabet^{k-j}$, for any 
$\tilde y_{-k}^{-j-1}\in \alphabet^{k-j}$, and for any
$\rho>0$ such that
$\rho\;R(\alpha,\,k,\,\rho)<\alpha $ we have
$$
\left(\frac{1-\rho\;R(\alpha,\,k,\,\rho)/\alpha}
{1+\rho\;R(\alpha,\,k,\,\rho)/\alpha}\right)^{2}
\; e^{-2\,\beta_{j,\,k}}\le 
\frac{\proba\big(Y_{0}=a\tq Y_{-j}^{-1}=y_{-j}^{-1},\,
Y_{-k}^{-j-1}=y_{-k}^{-j-1}\big)}
{\proba\big(Y_{0}=a\tq Y_{-j}^{-1}=y_{-j}^{-1},\,
Y_{-k}^{-j-1}=\tilde y_{-k}^{-j-1}\big)}
$$
$$
\le 
\left(\frac{1+\rho\;R(\alpha,\,k,\,\rho)/\alpha}
{1-\rho\;R(\alpha,\,k,\,\rho)/\alpha}\right)^{2}
\; e^{2\,\beta_{j,\,k}}
$$
We also have
$$
\left(\frac{1-\rho\;R(\alpha,\,k,\,\rho)/\alpha}
{1+\rho\;R(\alpha,\,k,\,\rho)/\alpha}\right)^{2}
\; e^{-2\,\beta_{j,\,k}}\le 
\frac{\proba\big(Y_{0}=a\tq Y_{-j}^{-1}=y_{-j}^{-1},\,
Y_{-k}^{-j-1}=y_{-k}^{-j-1}\big)}
{\proba\big(Y_{0}=a\tq Y_{-j}^{-1}=y_{-j}^{-1}\big)}
$$
$$
\le 
\left(\frac{1+\rho\;R(\alpha,\,k,\,\rho)/\alpha}
{1-\rho\;R(\alpha,\,k,\,\rho)/\alpha}\right)^{2}
\; e^{2\,\beta_{j,\,k}}\;,
$$
and
$$
\proba\big(Y_{0}=a\tq Y_{-j}^{-1}=y_{-j}^{-1}\big)\ge
\big(\alpha-\rho\;R(\alpha,\,j,\,\rho)\big)\;.
$$
\end{proposition}
The proof will be given in Section \ref{proofs}.

\section{Auxiliary results}\label{lemmas}

In this section we collect together some technical lemmas that will be
used in the proof of the main results. In what follows we will always
assume, without further mention, that Hypotheses H1 and H2 are fulfilled.

\begin{lemma}\label{nonj}
For any $k\ge j>0$ we have 
$$
\inf_{\substack{a\in \alphabet,\, x_{-j}^{-1}\in \alphabet^{j}\\
\eta_{-k}^{-j-1}\in \alphabet^{k-j}}}\quad
\proba\big(X_{0}=a\tq \paire_{-k}^{-j-1}=\eta_{-k}^{-j-1},\,
X_{-j}^{-1}=x_{-j}^{-1}\big)\ge\alpha\,,
$$
and in particular
$$
\inf_{\substack{a\in \alphabet,\, x_{-j}^{-1}\in \alphabet^{j}}}
\quad
\proba\big(X_{0}=a\tq X_{-j}^{-1}=x_{-j}^{-1}\big)\ge\alpha\;.
$$
\end{lemma}

\begin{proof}
By Bayes formula we have 
%\marginpar{par B2}
$$
\proba\big(X_{0}=a\tq \paire_{-k}^{-j-1}=\eta_{-k}^{-j-1},\,
X_{-j}^{-1}=x_{-j}^{-1}\big)
$$
$$
=\sum_{y_{-j}^{-1}}\proba\big(X_{0}=a\tq Y_{-j}^{-1}=y_{-j}^{-1},\,
\paire_{-k}^{-j-1}=\eta_{-k}^{-j-1},\,
X_{-j}^{-1}=x_{-j}^{-1}\big)\times
$$
$$
\proba\big(Y_{-j}^{-1}=y_{-j}^{-1}
\tq \paire_{-k}^{-j-1}=\eta_{-k}^{-j-1},\,
X_{-j}^{-1}=x_{-j}^{-1}\big)\, .
$$
Using Hypothesis H2 the result follows.
\end{proof}

\begin{lemma}\label{lescon2}

%%%%%%%%%%%%%%%%
For any $\rho\in (0,1)$, any 
$k>j\ge 0$, any $x_{-k}^{0}\in\alphabet^{k+1}$ and any
$w_{-k}^{-j-1}\in\alphabet^{k-j}$,  we have 
$$
\bigl| \proba\big(X_0=x_0\tq  
X_{-j-1}^{-1}=x_{-j-1}^{-1},
Y_{-k}^{-j-1}=w_{-k}^{-j-1}\big) -
\proba\big(X_0=x_0\tq X_{-k}^{-1}=x_{-k}^{-1}\big)\bigr|
$$
$$
\leq 
e^{\beta_{j+1,\,k}}-1\;,
$$
and 
$$
\bigl| \proba\big(X_0=x_0\tq  
X_{-j-1}^{-1}=x_{-j-1}^{-1},
Y_{-k}^{-j-2}=w_{-k}^{-j-2}\big) -
\proba\big(X_0=x_0\tq X_{-k}^{-1}=x_{-k}^{-1}\big)
\bigr|
$$
$$
\leq e^{\beta_{j+1,\,k}}-1\;.
$$
\end{lemma}
\begin{proof}
For any $k>j+1$ we have
$$
\proba\big(X_0=x_0\tq  
X_{-j-1}^{-1}=x_{-j-1}^{-1},\,
Y_{-k}^{-j-1}=w_{-k}^{-j-1}\big)=
$$
%\marginpar{par B3}
$$
\sum_{\tilde x_{-k}^{-j-2}}
\proba\big(X_0=x_0\tq  
X_{-j-1}^{-1}=x_{-j-1}^{-1},\,X_{-k}^{-j-2}=\tilde x_{-k}^{-j-2},\, 
Y_{-k}^{-j-1}=w_{-k}^{-j-1}\big)
\;\times
$$
$$
 \proba\big(X_{-k}^{-j-2}=\tilde x_{-k}^{-j-2}\tq  
X_{-j-1}^{-1}=x_{-j-1}^{-1},\,
Y_{-k}^{-j-1}=w_{-k}^{-j-1}\big)\;.
$$
We now have from the definition of $\beta_{j+1,\,k}$
$$
\sum_{{\tilde y}_{-k}^{-j-1}} \proba\big(X_0=x_0\tq 
X_{-k}^{-j-2}=\tilde x_{-k}^{-j-2},\,X_{-j-1}^{-1}=x_{-j-1}^{-1},\,
Y_{-k}^{-j-1}=w_{-k}^{-j-1}\big)\times
$$
$$
\proba\big(Y_{-k}^{-j-1}=
{\tilde y}_{-k}^{-j-1}\tq X_{-k}^{-1}=x_{-k}^{-1}\big)
$$
$$
\le e^{\beta_{j+1,\,k}}\sum_{{\tilde y}_{-k}^{-j-1}}
\proba\big(Y_{-k}^{-j-1}=
{\tilde y}_{-k}^{-j-1}\tq X_{-k}^{-1}=x_{-k}^{-1}\big)\;\times
$$
$$
\proba\big(X_0=x_0\tq 
X_{-k}^{-1}=x_{-k}^{-1},\,
Y_{-k}^{-j-1}={\tilde y}^{-j-1}\big)
= e^{\beta_{j+1,\,k}}\;\proba\big(X_0=x_0\tq X_{-k}^{-1}=x_{-k}^{-1}\big)\;.
$$
Therefore
$$
\proba\big(X_0=x_0\tq  
X_{-j-1}^{-1}=x_{-j-1}^{-1},\,
Y_{-k}^{-j-1}=w_{-k}^{-j-1}\big)
$$
$$
\le e^{\beta_{j+1,\,k}}\;\proba\big(X_0=x_0\tq  
X_{-k}^{-1}=x_{-k}^{-1}\big)\;\times
$$
$$
\sum_{\tilde x_{-k}^{-j-2}}\;
\proba\big(X_{-k}^{-j-2}=\tilde x_{-k}^{-j-2}\tq  
X_{-j-1}^{-1}=x_{-j-1}^{-1},\,
Y_{-k}^{-j-1}=w_{-k}^{-j-1}\big)
$$
$$
= e^{\beta_{j+1,\,k}}\;\proba\big(X_0=x_0\tq  
X_{-k}^{-1}=x_{-k}^{-1}\big)\;.
$$
We have similarly the lower bound
$$
\proba\big(X_0=x_0\tq  
X_{-j-1}^{-1}=x_{-j-1}^{-1},\,
Y_{-k}^{-j-1}=w_{-k}^{-j-1}\big)
$$
$$
\ge e^{-\beta_{j+1,\,k}}\;\proba\big(X_0=x_0\tq  
X_{-k}^{-1}=x_{-k}^{-1}\big)\;.
$$
Observing that
$$
1-e^{-\beta_{j+1,\,k}}=e^{-\beta_{j+1,\,k}}\;\left(e^{\beta_{j+1,\,k}}-1\right)\le
e^{\beta_{j+1,\,k}}-1\;,
$$
the lower bound follows.
For $k=j+1$ the estimation is similar and left to the reader.

To get the second result we write
$$
\proba\big(X_0=x_0\tq  
X_{-j-1}^{-1}=x_{-j-1}^{-1},
Y_{-k}^{-j-2}=w_{-k}^{-j-2}\big)
$$

$$
=\sum_{b\in\alphabet}
\proba\big(X_0=x_0\tq  
X_{-j-1}^{-1}=x_{-j-1}^{-1},
Y_{-k}^{-j-2}=w_{-k}^{-j-2},\, Y_{-j-1}=b\big)
$$
$$
\times \proba\big(Y_{-j-1}=b\tq  
X_{-j-1}^{-1}=x_{-j-1}^{-1},
Y_{-k}^{-j}=w_{-k}^{-j}\big)
$$
The result follows by applying the first estimate to each term in the
sum.
%\marginpar{utilise la premiere partie}
\end{proof}

\begin{lemma}\label{lescon3}
%%%%%%%%%%%%%%%%%%%%
For any $j\ge 0$, for  any $k> j+1$ and any $w_{-k}^0\in\alphabet^{k+1}$ we have 
\begin{equation*}
\proba\big(Y_{-j-1}=w_{-j-1}\tq X_{-j}^{-1}=w_{-j}^{-1},\,
Y_{-k}^{-j-2}=\,w_{-k}^{-j-2}\big)\;\ge\;(1-\rho)\;\alpha\;e^{-\somb_{j,\,k}}.
\end{equation*}
\end{lemma}
\begin{proof}
%\marginpar{peut etre mettre cela plus loin quand on s'en sert}
We have
$$
\proba\big(Y_{-j-1}=w_{-j-1},\, X_{-j}^{-1}=w_{-j}^{-1},\,
Y_{-k}^{-j-2}=\,w_{-k}^{-j-2}\big)
$$
$$
=\sum_{x_{-k}^{-j-1}}
\proba\big(
Y_{-j-1}=\,w_{-j-1},\,X_{-j}^{-1}=w_{-j}^{-1}\tq
X_{-k}^{-j-1}=x_{-k}^{-j-1},\,Y_{-k}^{-j-2}=\,w_{-k}^{-j-2}\big)\;\times
$$
$$
\proba\big(
X_{-k}^{-j-1}=x_{-k}^{-j-1},\,Y_{-k}^{-j-2}=\,w_{-k}^{-j-2}\big)
$$
$$
\ge \sum_{x_{-k}^{-j-2}}
\proba\big(X_{-j-1}=w_{-j-1},\,
X_{-k}^{-j-2}=x_{-k}^{-j-2},\,Y_{-k}^{-j-2}=\,w_{-k}^{-j-2}\big)\;\times
$$
$$
\proba\big(
Y_{-j-1}=\,w_{-j-1},\,X_{-j}^{-1}=w_{-j}^{-1}\tq 
$$
$$
X_{-j-1}=w_{-j-1},\,
X_{-k}^{-j-2}=x_{-k}^{-j-2},\,Y_{-k}^{-j-2}=\,w_{-k}^{-j-2}\big)\;.
$$
We have for any $(u,v)\in\alphabet^{2}$
$$
\proba\big(
Y_{-j-1}=\,w_{-j-1},\,X_{-j}^{-1}=w_{-j}^{-1}
\tq 
X_{-k}^{-j-1}=x_{-k}^{-j-1},\,Y_{-k}^{-j-2}=\,w_{-k}^{-j-2}\big)
$$
$$
=\proba\big(
Y_{-j-1}=\,w_{-j-1}
\tq 
X_{-k}^{-j-1}=x_{-k}^{-j-1},\,Y_{-k}^{-j-2}=\,w_{-k}^{-j-2}\big)\;\times
$$
$$
\prod_{s=-j}^{-1}\;\proba\big(X_{s}=w_{s}\tq
X_{-j}^{-s-1}=w_{-j}^{-s-1},\,
Y_{-j-1}=\,w_{-j-1},\,
X_{-j-1}=x_{-j-1},\,
$$
$$
\paire_{-k}^{-j-2}=(x_{-k}^{-j-2},w_{-k}^{-j-2})\big)
$$
$$
\ge e^{-\somb_{j,\,k}}\;\proba\big(
Y_{-j-1}=\,w_{-j-1}
\tq
X_{-k}^{-j-1}=x_{-k}^{-j-1},\,Y_{-k}^{-j-2}=\,w_{-k}^{-j-2}\big)\;\times
$$
$$
\prod_{s=-j}^{-1}\!\!\!\proba\big(X_{s}=w_{s}\tq
X_{-j}^{-s-1}=w_{-j}^{-s-1},
Y_{-j-1}=\,u,
X_{-j-1}=v,
\paire_{-k}^{-j-2}=(x_{-k}^{-j-2},w_{-k}^{-j-2})
\big).
$$
If $x_{-j-1}=w_{-j-1}$, we get using the
definition of $\rho$
$$
\proba\big(
Y_{-j-1}=\,w_{-j-1},\,X_{-j}^{-1}=w_{-j}^{-1}
\tq 
X_{-k}^{-j-1}=x_{-k}^{-j-1},\,Y_{-k}^{-j-2}=\,w_{-k}^{-j-2}\big)
$$
$$
\ge e^{-\somb_{j,\,k}}\;(1-\rho)
\prod_{s=-j}^{-1}\;\proba\big(X_{s}=w_{s}\tq
X_{-j}^{-s-1}=w_{-j}^{-s-1},\,
Y_{-j-1}=u,\,
X_{-j-1}=v,
$$
$$
\paire_{-k}^{-j-2}=(x_{-k}^{-j-2},w_{-k}^{-j-2})\,
\big)
$$
$$
= e^{-\somb_{j,\,k}}\;(1-\rho)\;
\proba\big(X_{-j}^{-1}=w_{-j}^{-1}\tq
Y_{-j-1}=u,\,
X_{-j-1}=v,
\paire_{-k}^{-j-2}=\eta_{-k}^{-j-2}
\big)\;.
$$
We can write
$$
\proba\big(
Y_{-j-1}=w_{-j-1},\,X_{-j}^{-1}=w_{-j}^{-1}
\tq 
X_{-j-1}=w_{-j-1},
X_{-k}^{-j-2}=x_{-k}^{-j-2},Y_{-k}^{-j-2}=w_{-k}^{-j-2}\big)
$$
$$
=\sum_{(u,v)\in\alphabet^{2}}\!\!\!\!\!
\proba\big(Y_{-j-1}=u,\,
X_{-j-1}=v\tq
 \paire_{-k}^{-j-2}=(x_{-k}^{-j-2},w_{-k}^{-j-2})\big)\, \times
$$
$$
\proba\big(
Y_{-j-1}=w_{-j-1},\,X_{-j}^{-1}=w_{-j}^{-1}
\,\big|
X_{-j-1}=w_{-j-1},\,X_{-k}^{-j-2}=x_{-k}^{-j-2},\,Y_{-k}^{-j-2}=w_{-k}^{-j-2}\big)
$$
$$
\ge e^{-\somb_{j,\,k}}\;(1-\rho)\;\sum_{(u,v)\in\alphabet^{2}}
\proba\big(Y_{-j-1}=u,\,
X_{-j-1}=v\tq \paire_{-k}^{-j-2}=(x_{-k}^{-j-2},w_{-k}^{-j-2})
\big)\;\times
$$
$$
\proba\big(X_{-j}^{-1}=w_{-j}^{-1}\tq
Y_{-j-1}=u,\,
X_{-j-1}=v,
\paire_{-k}^{-j-2}=(x_{-k}^{-j-2},w_{-k}^{-j-2})
\big)
$$
%\marginpar{Par B3}
$$
=e^{-\somb_{j,\,k}}\;(1-\rho)\;\proba\big(X_{-j}^{-1}=w_{-j}^{-1}\tq
\paire_{-k}^{-j-2}=(x_{-k}^{-j-2},w_{-k}^{-j-2})\big)
$$
We have by Hypothesis H2
$$
\proba\big(
X_{-k}^{-j-1}=x_{-k}^{-j-1},\,Y_{-k}^{-j-2}=\,w_{-k}^{-j-2}\big)
\ge \alpha\;\proba\big(
X_{-k}^{-j-2}=x_{-k}^{-j-2},\,Y_{-k}^{-j-2}=\,w_{-k}^{-j-2}\big).
$$
Combining the above estimates we get
$$
\proba\big(Y_{-j-1}=w_{-j-1},\, X_{-j}^{-1}=w_{-j}^{-1},\,
Y_{-k}^{-j-2}=\,w_{-k}^{-j-2}\big)
$$
%\marginpar{juste mais dit un peu vite dans l'autre version}
$$
\ge e^{-\somb_{j,\, k}}\;\alpha\;(1-\rho)
\sum_{x_{-k}^{-j-2}}
\proba\big(X_{-j}^{-1}=w_{-j}^{-1}\tq
\paire_{-k}^{-j-2}=(x_{-k}^{-j-2},w_{-k}^{-j-2})\big)\;\times
$$
$$
\proba\big(
X_{-k}^{-j-2}=x_{-k}^{-j-2},\,Y_{-k}^{-j-2}=\,w_{-k}^{-j-2}\big)
$$
$$
=e^{-\somb_{j,\,k}}\;\alpha\;(1-\rho)\;
\proba\big(X_{-j}^{-1}=w_{-j}^{-1},\,
\,Y_{-k}^{-j-2}=\,w_{-k}^{-j-2}\big)\;,
$$
and the result follows.
\end{proof}

\begin{lemma}\label{lescon4}
For any $\rho\in (0,1)$, any $k>j\ge 0$ and any $w_{-k}^0\in\alphabet^{k+1}$,
\begin{equation*}
\proba\big(X_{-j-1}\neq w_{-j-1}\tq
\,X_{-j}^{-1}=\,w_{-j}^{-1},\,
Y_{-k}^{-j-1}=\,w_{-k}^{-j-1}\big)
\;\le\;  \frac{\rho\,e^{2\,\somb_{j,\,k}}}{\alpha\,(1-\rho)^{2}}\,.
\end{equation*}
We have also 
\begin{equation*}
\proba\big(Y_{-j-1}\neq w_{-j-1}\tq
\,X_{-j-1}^{-1}=\,w_{-j-1}^{-1},\,
Y_{-k}^{-j-2}=\,w_{-k}^{-j-2}\big)
\;\le\; \rho\;e^{\somb_{j,\,k}}\;.
\end{equation*}
\end{lemma}

\begin{proof}
We have using the definition of $\rho$ %\marginpar{par B2 nfois}
$$
\proba\big(
Y_{-j-1}=\,w_{-j-1},\,X_{-j}^{-1}=w_{-j}^{-1}
\tq 
Y_{-k}^{-j-2}=\,w_{-k}^{-j-2},\,
X_{-k}^{-j-1}=x_{-k}^{-j-1}\big)
$$
$$
=\proba\big(
Y_{-j-1}=\,w_{-j-1}
\tq 
Y_{-k}^{-j-2}=\,w_{-k}^{-j-2},\,
X_{-k}^{-j-1}=x_{-k}^{-j-1}\big)\;\times
$$
$$
\prod_{s=-j}^{-1}\!\!\proba\big(X_{s}=w_{s}\tq
Y_{-j-1}=\,w_{-j-1},\,X_{-j}^{s-1}=w_{-j}^{s-1},\,
Y_{-k}^{-j-2}=\,w_{-k}^{-j-2},\,
X_{-k}^{-j-1}=x_{-k}^{-j-1}\big)
$$
$$
\le \rho\;\prod_{s=-j}^{-1}\proba\big(X_{s}=w_{s}\tq
$$
$$
Y_{-j-1}=\,w_{-j-1},\,X_{-j}^{s-1}=w_{-j}^{s-1},\,
Y_{-k}^{-j-2}=\,w_{-k}^{-j-2},\,
X_{-k}^{-j-1}=x_{-k}^{-j-1}\big)\;.
$$
This quantity is bounded above by
$$
\rho\; e^{\somb_{j,\,k}} \;\prod_{s=-j}^{-1}\proba\big(X_{s}=w_{s}\tq
Y_{-j-1}=\,x_{-j-1},\,X_{-j}^{s-1}=w_{-j}^{s-1},\,
$$
$$
Y_{-k}^{-j-2}=\,w_{-k}^{-j-2},\,
X_{-k}^{-j-1}=x_{-k}^{-j-1}\big)
$$
$$
=\rho\; e^{\somb_{j,\,k}}\; \proba\big(X_{-j}^{-1}=w_{-j}^{-1}\tq
Y_{-j-1}=\,x_{-j-1},\,
\,Y_{-k}^{-j-2}=\,w_{-k}^{-j-2},\,
X_{-k}^{-j-1}=x_{-k}^{-j-1}\big)
$$
%\marginpar{par B2}
$$
=\frac{\rho\; e^{\somb_{j,\,k}}\;
\proba\big(X_{-j}^{-1}=w_{-j}^{-1},\,
Y_{-j-1}=\,x_{-j-1}\tq
\,Y_{-k}^{-j-2}=\,w_{-k}^{-j-2},\,
X_{-k}^{-j-1}=x_{-k}^{-j-1}\big)
}{\proba\big(
Y_{-j-1}=\,x_{-j-1}\tq 
Y_{-k}^{-j-2}=\,w_{-k}^{-j-2},\,
X_{-k}^{-j-1}=x_{-k}^{-j-1}\big)}
$$
$$
\le \frac{\rho\; e^{\somb_{j,\,k}}}{1-\rho}\;
\proba\big(X_{-j}^{-1}=w_{-j}^{-1},\,
Y_{-j-1}=\,x_{-j-1}\tq
\,Y_{-k}^{-j-2}=\,w_{-k}^{-j-2},\,
X_{-k}^{-j-1}=x_{-k}^{-j-1}\big)
$$
using the definition of $\rho$ and hypothesis H1.
This last quantity is obviously bounded above by
$$
\frac{\rho\; e^{\somb_{j,\,k}}}{1-\rho}\;\proba\big(X_{-j}^{-1}=w_{-j}^{-1}
\tq
Y_{-k}^{-j-2}=\,w_{-k}^{-j-2},\,
X_{-k}^{-j-1}=x_{-k}^{-j-1}\big)
$$
and we get
$$
\proba\big(X_{-j-1}\neq w_{-j-1},\,
\,X_{-j}^{-1}=\,w_{-j}^{-1},\,
Y_{-k}^{-j-1}=\,w_{-k}^{-j-1}\big)=
$$
%\marginpar{par def proba cond}
% \le \frac{\rho\; e^{\somb_{j,\,k}}}{1-\rho}\;
$$
\sum_{x_{-k}^{-j-2},\,x_{-j-1}\neq w_{-j-1}}\!\!\!\!\!\!\!\!\!\!\!\!\!\!\!\!
\proba\big(X_{-j}^{-1}=w_{-j}^{-1},\,Y_{-j-1}=w_{-j-1}
\tq
Y_{-k}^{-j-2}=\,w_{-k}^{-j-2},\,
X_{-k}^{-j-1}=x_{-k}^{-j-1}\big)
$$
$$
\times \;\proba\big(
X_{-k}^{-j-1}=\,x_{-k}^{-j-1},\,Y_{-k}^{-j-2}=\,w_{-k}^{-j-2}\big)\;.
$$
$$
\le \frac{\rho\; e^{\somb_{j,\,k}}}{1-\rho}\;
\sum_{x_{-k}^{-j-1},\,x_{-j-1}\neq w_{-j-1}}
\proba\big(X_{-j}^{-1}=w_{-j}^{-1},\,Y_{-k}^{-j-2}=\,w_{-k}^{-j-2},\,
X_{-k}^{-j-1}=x_{-k}^{-j-1}\big)
$$
$$
\le \frac{\rho\; e^{\somb_{j,\,k}}}{1-\rho}\;
\proba\big(X_{-j}^{-1}=w_{-j}^{-1},\,Y_{-k}^{-j-2}=\,w_{-k}^{-j-2}\big)\;.
$$
We have obtained the bound
$$
\proba\big(X_{-j-1}\neq w_{-j-1},\,Y_{-j-1}=w_{-j-1}\tq
\,X_{-j}^{-1}=w_{-j}^{-1},\,
Y_{-k}^{-j-2}=w_{-k}^{-j-2}\big)
\le 
\frac{\rho\; e^{\somb_{j,\,k}}}{1-\rho}\;.
$$
Using  Lemma \ref{lescon3} and hypothesis H1
%\marginpar{par B2}
the first result follows.

In order to prove the second result, we start with the identity
$$
\proba\big(Y_{-j-1}\neq w_{-j-1},
\,X_{-j-1}^{-1}=\,w_{-j-1}^{-1},\,
Y_{-k}^{-j-2}=\,w_{-k}^{-j-2}\big)
=\sum_{c\neq w_{-j-1}} \sum_{x_{-k}^{-j-2}}
$$
$$
\proba\big(Y_{-j-1}=c,
\,X_{-j-1}^{-1}=\,w_{-j-1}^{-1},\,X_{-k}^{-j-2}=\,x_{-k}^{-j-2},\,
Y_{-k}^{-j-2}=\,w_{-k}^{-j-2}\big)
$$
%\marginpar{par B2}
$$
=\sum_{c\neq w_{-j-1}} \sum_{x_{-k}^{-j-2}}
\proba\big(X_{-j-1}=\,w_{-j-1},\,X_{-k}^{-j-2}=\,x_{-k}^{-j-2},\,
Y_{-k}^{-j-2}=\,w_{-k}^{-j-2}\big)\;\times
$$
$$
\proba\big(Y_{-j-1}=c,
\,X_{-j}^{-1}=w_{-j}^{-1}\tq 
X_{-j-1}=\,w_{-j-1},\,X_{-k}^{-j-2}=x_{-k}^{-j-2},\,
Y_{-k}^{-j-2}=w_{-k}^{-j-2}\big).
$$
We have
$$
\proba\big(Y_{-j-1}=c,
\,X_{-j}^{-1}=\,w_{-j}^{-1}\tq 
X_{-j-1}=\,w_{-j-1},\,X_{-k}^{-j-2}=x_{-k}^{-j-2},\,
Y_{-k}^{-j-2}=w_{-k}^{-j-2}\big)
$$
$$
=\proba\big(Y_{-j-1}=c
\tq 
X_{-j-1}=\,w_{-j-1},\,X_{-k}^{-j-2}=\,x_{-k}^{-j-2},\,
Y_{-k}^{-j-2}=\,w_{-k}^{-j-2}\big)
\;\times
$$
$$
\prod_{s=-j}^{-1} \proba\big(X_{s}=w_{s}\tq Y_{-j-1}=c,
\,X_{-j}^{s-1}=\,w_{-j}^{s-1},
$$
$$
X_{-j-1}=\,w_{-j-1},\,X_{-k}^{-j-2}=\,x_{-k}^{-j-2},\,
Y_{-k}^{-j-2}=\,w_{-k}^{-j-2}\big)
$$
$$
\le e^{\somb_{j,\,k}}\; \proba\big(Y_{-j-1}=c
\tq 
X_{-j-1}=\,w_{-j-1},\,X_{-k}^{-j-2}=\,x_{-k}^{-j-2},\,
Y_{-k}^{-j-2}=\,w_{-k}^{-j-2}\big)
\;\times
$$
$$
\prod_{s=-j}^{-1} \proba\big(X_{s}=w_{s}\tq Y_{-j-1}=u,
\,X_{-j}^{s-1}=\,w_{-j}^{s-1}\;,
$$
$$
X_{-j-1}=\,w_{-j-1},\,X_{-k}^{-j-2}=\,x_{-k}^{-j-2},\,
Y_{-k}^{-j-2}=\,w_{-k}^{-j-2}\big)
$$
for any $u\in\alphabet$.

Using the definition of $\rho$ we get 
$$
\sum_{c\neq w_{-j-1}}\proba\big(Y_{-j-1}=c,
\,X_{-j}^{-1}=\,w_{-j}^{-1}\tq 
$$
$$
X_{-j-1}=\,w_{-j-1},\,X_{-k}^{-j-2}=\,x_{-k}^{-j-2},\,
Y_{-k}^{-j-2}=\,w_{-k}^{-j-2}\big)
$$
$$
\le \rho\;e^{\somb_{j,\,k}}\;\prod_{s=-j}^{-1} 
\proba\big(X_{s}=w_{s}\tq Y_{-j-1}=u,
\,X_{-j}^{s-1}=\,w_{-j}^{s-1},\,
$$
$$
X_{-j-1}=\,w_{-j-1},\,X_{-k}^{-j-2}=\,x_{-k}^{-j-2},\,
Y_{-k}^{-j-2}=\,w_{-k}^{-j-2}\big)
$$
$$
=\rho\;e^{\somb_{j,\,k}}\;\proba\big(X_{-j}^{-1}=w_{-j}^{-1}\tq
$$
$$
Y_{-j-1}=u,\,
X_{-j-1}=\,w_{-j-1},\,X_{-k}^{-j-2}=\,x_{-k}^{-j-2},\,
Y_{-k}^{-j-2}=\,w_{-k}^{-j-2}\big)\;.
$$
We obviously have
%\marginpar{ATTENTIOn les Y oublies dans l'autre version}
$$
\sum_{c\neq w_{-j-1}}\proba\big(Y_{-j-1}=c,
\,X_{-j}^{-1}=\,w_{-j}^{-1}\tq 
$$
$$
X_{-j-1}=\,w_{-j-1},\,X_{-k}^{-j-2}=\,x_{-k}^{-j-2},\,
Y_{-k}^{-j-2}=\,w_{-k}^{-j-2}\big)
$$
$$
=\sum_{u}
\proba\big(Y_{-j-1}=u\tq
X_{-j-1}=\,w_{-j-1},\,X_{-k}^{-j-2}=\,x_{-k}^{-j-2},\,
Y_{-k}^{-j-2}=\,w_{-k}^{-j-2}\big)\;\times
$$
$$
\sum_{c\neq w_{-j-1}}\proba\big(Y_{-j-1}=c,
\,X_{-j}^{-1}=\,w_{-j}^{-1}\tq
$$
$$ 
X_{-j-1}=\,w_{-j-1},\,X_{-k}^{-j-2}=\,x_{-k}^{-j-2},\,
Y_{-k}^{-j-2}=\,w_{-k}^{-j-2}\big)
$$
$$
\le \rho\;e^{\somb_{j,\,k}}\;\sum_{u}
\proba\big(Y_{-j-1}=u\tq 
X_{-j-1}=\,w_{-j-1},\,X_{-k}^{-j-2}=\,x_{-k}^{-j-2},\,
Y_{-k}^{-j-2}=\,w_{-k}^{-j-2}\big)\,\times
$$
$$
\proba\big(X_{-j}^{-1}=w_{-j}^{-1}\tq Y_{-j-1}=u,\,
X_{-j-1}=\,w_{-j-1},\,X_{-k}^{-j-2}=\,x_{-k}^{-j-2},\,
Y_{-k}^{-j-2}=\,w_{-k}^{-j-2}\big)
$$
%\marginpar{par B2}
$$
=\rho\;e^{\somb_{j,\,k}}\;\sum_{u}
\proba\big(X_{-j}^{-1}=w_{-j}^{-1},\,Y_{-j-1}=u\tq
$$
$$
X_{-j-1}=\,w_{-j-1},\,X_{-k}^{-j-2}=\,x_{-k}^{-j-2},\,
Y_{-k}^{-j-2}=\,w_{-k}^{-j-2}\big)
$$
$$
=\rho\;e^{\somb_{j,\,k}}\;
\proba\big(X_{-j}^{-1}=w_{-j}^{-1}\tq
X_{-j-1}=\,w_{-j-1},\,X_{-k}^{-j-2}=\,x_{-k}^{-j-2},\,
Y_{-k}^{-j-2}=\,w_{-k}^{-j-2}\big)\;.
$$
Therefore
$$
\proba\big(Y_{-j-1}\neq w_{-j-1},
\,X_{-j-1}^{-1}=\,w_{-j-1}^{-1},\,
Y_{-k}^{-j-2}=\,w_{-k}^{-j-2}\big)
$$
$$
\le \rho\;e^{\somb_{j,\,k}}\!\!\!\sum_{x_{-k}^{-j-2}}\!\!
\proba\big(X_{-j}^{-1}=w_{-j}^{-1}\tq
X_{-j-1}=\,w_{-j-1},\,X_{-k}^{-j-2}=\,x_{-k}^{-j-2},\,
Y_{-k}^{-j-2}=\,w_{-k}^{-j-2}\big)\,\times
$$
$$
\proba\big(X_{-j-1}=\,w_{-j-1},\,X_{-k}^{-j-2}=\,x_{-k}^{-j-2},\,
Y_{-k}^{-j-2}=\,w_{-k}^{-j-2}\big)
$$
$$
=\rho\;e^{\somb_{j,\,k}}\!\!\sum_{x_{-k}^{-j-2}}
\proba\big(X_{-j}^{-1}=w_{-j}^{-1},\,
X_{-j-1}=\,w_{-j-1},\,X_{-k}^{-j-2}=\,x_{-k}^{-j-2},\,
Y_{-k}^{-j-2}=\,w_{-k}^{-j-2}\big)
$$
$$
=\rho\;e^{\somb_{j,\,k}}\;\proba\big(X_{-j}^{-1}=w_{-j}^{-1},\,
X_{-j-1}=\,w_{-j-1},\,Y_{-k}^{-j-2}=\,w_{-k}^{-j-2}\big)\;,
$$
and the second result  result follows.
\end{proof}

%\marginpar{C'etait assez trivial dans \cite{cgl}, ici il faut travailler.}

\begin{lemma}\label{undeplus}
For any $k\ge0$
$$
\sup_{a\in\alphabet}\;\sup_{w_{-k}^{-1}\in\alphabet^{k}}
\big|\proba\big(Y_{0}=a\tq Y_{-k}^{-1}= w_{-k}^{-1}\big)-
\proba\big(X_{0}=a\tq Y_{-k}^{-1}=w_{-k}^{-1}\big)\big|\le \rho\;.
$$
\end{lemma}
\begin{proof}
We write
$$
\proba\big(Y_{0}=a,\,Y_{-k}^{-1}=w_{-k}^{-1}\big)
=\sum_{x_{-k}^{0}}
\proba\big(Y_{0}=a,\,Y_{-k}^{-1}=w_{-k}^{-1},\,X_{-k}^{0}=x_{-k}^{0}\big)
$$
$$
=\sum_{x_{-k}^{0}}
\proba\big(Y_{0}=a
\tq 
Y_{-k}^{-1}=w_{-k}^{-1},\,X_{-k}^{0}=x_{-k}^{0}\big)\;
\proba\big(X_{-k}^{0}=x_{-k}^{0},\,Y_{-k}^{-1}=w_{-k}^{-1}\big)\;.
$$
We will split this sum in two sums, one with $x_{0}=a$ and the other
one with $x_{0}\neq a$. 

If $x_{0}=a$ we have
$$
\proba\big(Y_{0}=a
\tq 
Y_{-k}^{-1}=w_{-k}^{-1},\,X_{-k}^{0}=x_{-k}^{0}\big)
$$
$$
=1-\sum_{w_{0}\neq a}\proba\big(Y_{0}=w_{0}
\tq Y_{-k}^{-1}=w_{-k}^{-1},\,X_{-k}^{0}=x_{-k}^{0}\big)\ge 1-\rho
$$
from the definition of $\rho$. Therefore
$$
\sum_{x_{-k}^{-1}}
\proba\big(Y_{0}=a
\tq
Y_{-k}^{-1}=w_{-k}^{-1},\,X_{-k}^{-1}=x_{-k}^{-1},\,X_{0}=a\big)\;\times
$$
$$
\proba\big(X_{-k}^{-1}=x_{-k}^{-1},\,X_{0}=a,\,Y_{-k}^{-1}=w_{-k}^{-1}\big)
$$
$$
\ge(1-\rho)\;\sum_{x_{-k}^{-1}}
\proba\big(X_{-k}^{-1}=x_{-k}^{-1},\,X_{0}=a,\,Y_{-k}^{-1}=w_{-k}^{-1}
\big)
=(1-\rho)\;
\proba\big(X_{0}=a,\,Y_{-k}^{-1}=w_{-k}^{-1}\big).
$$
We conclude that
$$
\proba\big(Y_{0}=a,\,Y_{-k}^{-1}=w_{-k}^{-1}\big)\ge (1-\rho)\;
\proba\big(X_{0}=a,\,Y_{-k}^{-1}=w_{-k}^{-1}\big)\;,
$$
which implies
$$
\proba\big(Y_{0}=a\tq\,Y_{-k}^{-1}=w_{-k}^{-1}\big)\ge (1-\rho)\;
\proba\big(X_{0}=a\tq\,Y_{-k}^{-1}=w_{-k}^{-1}\big)\;,
$$
and therefore
\begin{equation}\label{lbb}
\proba\big(Y_{0}=a\tq\,Y_{-k}^{-1}=w_{-k}^{-1}\big)\ge \;
\proba\big(X_{0}=a\tq\,Y_{-k}^{-1}=w_{-k}^{-1}\big)-\rho\;.
\end{equation}
We also have the upper bound for $x_{0}=a$
$$
\sum_{x_{-k}^{-1}}
\proba\big(Y_{0}=a
\tq
Y_{-k}^{-1}=w_{-k}^{-1},\,X_{-k}^{-1}=x_{-k}^{-1},\,X_{0}=a\big)\;\times
$$
$$
\proba\big(X_{-k}^{-1}=x_{-k}^{-1},\,X_{0}=a,\,Y_{-k}^{-1}=w_{-k}^{-1}\big)
$$
%\marginpar{par B2}
$$
\le \sum_{x_{-k}^{-1}}
\proba\big(X_{-k}^{-1}=x_{-k}^{-1},\,X_{0}=a,\,Y_{-k}^{-1}=w_{-k}^{-1}
\big)
=\proba\big(X_{0}=a,\,Y_{-k}^{-1}=w_{-k}^{-1}\big)\;.
$$
For $x_{0}\neq a$ we have from the definition of $\rho$
$$
\sum_{x_{-k}^{-1},\,x_{0}\neq a}
\proba\big(Y_{0}=a
\tq Y_{-k}^{-1}=w_{-k}^{-1},\,X_{-k}^{0}=x_{-k}^{0}\big)\;
\proba\big(X_{-k}^{0}=x_{-k}^{0},\,Y_{-k}^{-1}=w_{-k}^{-1}
\big)
$$
$$
\le \rho \sum_{x_{-k}^{-1},\,x_{0}\neq a}
 \proba\big(X_{-k}^{0}=x_{-k}^{0},\,Y_{-k}^{-1}=w_{-k}^{-1}
\big)
\le \rho \;\proba\big(Y_{-k}^{-1}=w_{-k}^{-1}\big)\;.
$$
From the two last  estimates we get 
$$
\proba\big(Y_{0}=a,\,Y_{-k}^{-1}=w_{-k}^{-1}\big)\le
\proba\big(X_{0}=a,\,Y_{-k}^{-1}=w_{-k}^{-1}\big) 
+\rho \;\proba\big(Y_{-k}^{-1}=w_{-k}^{-1}\big)\;,
$$
hence
$$
\proba\big(Y_{0}=a\tq Y_{-k}^{-1}=w_{-k}^{-1}\big)\le
\proba\big(X_{0}=a\tq Y_{-k}^{-1}=w_{-k}^{-1}\big)
+\rho\;,
$$
and the result follows using the lower bound \eqref{lbb}. 
\end{proof}

\section{Proofs}\label{proofs}

\begin{proof}[Proof of Proposition \ref{leX}]
The non-nullness follows from Lemma
\ref{nonj}.
%\marginpar{demonstration anterieure incorrecte.}

We also have
$$
\proba\big(X_0=a\tq X_{-k}^{-1}=x_{-k}^{-1}\big)=
$$
%\marginpar{par B3}
$$
\sum_{y_{-k}^{-j-1}}
\proba\big(X_0=a\tq X_{-j}^{-1}=x_{-j}^{-1},\,
\paire_{-k}^{-j-1}=(x_{-k}^{-j-1},y_{-k}^{-j-1})\big)\times
%%\marginpar{voir cette notation $\proba(Z=dz)$.}
$$
$$
\proba\big(Y_{-k}^{-j-1}=y_{-k}^{-j-1}\tq
X_{-k}^{-1}=x_{-k}^{-1}\big)\;.
$$
We now fix a sequence
$\zeta_{-k}^{-j-1}\in(\alphabet^{2})^{k-j}$. 

We deduce that  for any $a$ and any $x_{-k}^{-1}$ 
$$
e^{-\beta_{j,\,k}}\;
\proba\big(X_0=a\tq X_{-j}^{-1}=x_{-j}^{-1},\,
\paire_{-k}^{-j-1}=\zeta_{-k}^{-j-1}\big)
\le \proba\big(X_0=a\tq X_{-k}^{-1}=x_{-k}^{-1}\big)\le
$$
$$
e^{\beta_{j,\,k}}\;
\proba\big(X_0=a\tq X_{-j}^{-1}=x_{-j}^{-1},\,
\paire_{-k}^{-j-1}=\zeta_{-k}^{-j-1}\big)
$$
and the second result follows.
\end{proof}
\begin{proof}[Proof of Theorem \ref{blurred}.]
We first observe that from Lemma \ref{undeplus} it is enough to
establish an upper bound on
$$
\big|\proba\big(X_{0}=a\tq Y_{-k}^{-1}=w_{-k}^{-1}\big) -
\proba\big(X_{0}=a\tq X_{-k}^{-1}=w_{-k}^{-1}\big) \big|\;.
$$
For $k=0$ this quantity is equal to zero and therefore we will from
now on assume $k\ge1$.

We write
$$
\proba\big(X_{0}=a\tq Y_{-k}^{-1}=w_{-k}^{-1}\big) -
\proba\big(X_{0}=a\tq X_{-k}^{-1}=w_{-k}^{-1}\big)
$$
$$
=\sum_{j=0}^{k-1}\bigg[
\proba\big(X_{0}=a\tq X_{-j}^{-1}=w_{-j}^{-1},\,
Y_{-k}^{-j-1}=w_{-k}^{-j-1}\big)
$$
%\marginpar{par B2}
$$
-\proba\big(X_{0}=a\tq X_{-j-1}^{-1}=w_{-j-1}^{-1},\,
Y_{-k}^{-j-2}=w_{-k}^{-j-2}\big)
\bigg]\;.
$$
We have for $0\le j\le k-1$
$$
\proba\big(X_{0}=a\tq X_{-j}^{-1}=w_{-j}^{-1},\,
Y_{-k}^{-j-1}=w_{-k}^{-j-1}\big)
$$
$$
-\proba\big(X_{0}=a\tq X_{-j-1}^{-1}=w_{-j-1}^{-1},\,
Y_{-k}^{-j-2}=w_{-k}^{-j-2}\big)
$$
$$
=\sum_{x_{-j-1}}\big[
\proba\big(X_{0}=a\tq X_{-j}^{-1}=w_{-j}^{-1},\,X_{-j-1}=x_{-j-1},\,
Y_{-k}^{-j-1}=w_{-k}^{-j-1}\big)
$$
$$-\proba\big(X_{0}=a\tq X_{-j-1}^{-1}=w_{-j-1}^{-1},\,
Y_{-k}^{-j-2}=w_{-k}^{-j-2}\big)
\big]\;
\times 
$$
$$
\proba\big(X_{-j-1}=x_{-j-1}\tq X_{-j}^{-1}=w_{-j}^{-1},\,
Y_{-k}^{-j-1}=w_{-k}^{-j-1}\big)\;.
$$
For $x_{-j-1}\neq w_{-j-1}$,
we apply Lemma \ref{lescon2} to each term in the square brackets, we
get
$$
\bigg|\sum_{x_{-j-1}\neq w_{-j-1}}\big[
\proba\big(X_{0}=a\tq X_{-j}^{-1}=w_{-j}^{-1},\,X_{-j-1}=x_{-j-1},\,
Y_{-k}^{-j-1}=w_{-k}^{-j-1}\big)
$$
$$-\proba\big(X_{0}=a\tq X_{-j-1}^{-1}=w_{-j-1}^{-1},\,
Y_{-k}^{-j-2}=w_{-k}^{-j-2}\big)
\big]
$$
$$
\times \;
\proba\big(X_{-j-1}=x_{-j-1}\tq X_{-j}^{-1}=w_{-j}^{-1},\,
Y_{-k}^{-j-1}=w_{-k}^{-j-1}\big)
\bigg|
$$
$$
\le  2\,\big(e^{\beta_{j+1,\,k}}-1\big)\;
\sum_{x_{-j-1}\neq w_{-j-1}}\!\!\!\!\!\!
\proba\big(X_{-j-1}=x_{-j-1}\tq X_{-j}^{-1}=w_{-j}^{-1},\,
Y_{-k}^{-j-1}=w_{-k}^{-j-1}\big)+
$$
$$
\sum_{x_{-j-1}\neq w_{-j-1}}
\proba\big(X_{-j-1}=x_{-j-1}\tq X_{-j}^{-1}=w_{-j}^{-1},\,
Y_{-k}^{-j-1}=w_{-k}^{-j-1}\big)\;\times
$$
$$
\bigg|\proba\big(X_{0}=a\tq
X_{-k}^{-j-2}=x_{-k}^{-j-2},\,X_{-j-1}=x_{-j-1},\,
X_{-j}^{-1}=w_{-j}^{-1}\big)
$$
$$
-\proba\big(X_{0}=a\tq X_{-k}^{-j-2}=x_{-k}^{-j-2},
\,X_{-j-1}=w_{-j-1},\, X_{-j}^{-1}=w_{-j}^{-1}\big)\bigg|
$$
$$
\le
\left[2\,\big(e^{\beta_{j+1,\,k}}-1\big)+
\big(e^{2\,\beta_{j+1,\,k}}-1\big)\right]\;\times
$$
$$
\proba\big(X_{-j-1}\neq w_{-j-1}\tq X_{-j}^{-1}=w_{-j}^{-1},\,
Y_{-k}^{-j-1}=w_{-k}^{-j-1}\big)
$$
$$
\le \frac{\rho\;e^{2\,\somb_{j,\,k}}\;
\left[2\,\big(e^{\beta_{j,\,k}}-1\big)+\big(e^{2\,\beta_{j,\,k}}-1\big)\right]}{
\alpha\;(1-\rho)^{2}}\;,
$$
by Proposition \ref{leX} and Lemma \ref{lescon4}.

We now consider the case $x_{-j-1}=w_{-j-1}$. We have to estimate
$$
 \bigg|\proba\big(X_{0}=a\tq X_{-j-1}^{-1}=w_{-j-1}^{-1}
Y_{-k}^{-j-1}=w_{-k}^{-j-1}\big)
$$
$$
-\proba\big(X_{0}=a\tq X_{-j-1}^{-1}=w_{-j-1}^{-1},\,
Y_{-k}^{-j-2}=w_{-k}^{-j-2}\big)
\bigg|\;.
$$
We write
$$
\proba\big(X_{0}=a\tq X_{-j-1}^{-1}=w_{-j-1}^{-1}
Y_{-k}^{-j-1}=w_{-k}^{-j-1}\big)
$$
$$
-\proba\big(X_{0}=a\tq X_{-j-1}^{-1}=w_{-j-1}^{-1},\,
Y_{-k}^{-j-2}=w_{-k}^{-j-2}\big)
$$
$$
=\sum_{c}
\big[\proba\big(X_{0}=a\tq X_{-j-1}^{-1}=w_{-j-1}^{-1}
Y_{-k}^{-j-1}=w_{-k}^{-j-1}\big)
$$
%\marginpar{par B2}
$$
-\proba\big(X_{0}=a\tq X_{-j-1}^{-1}=w_{-j-1}^{-1},\,Y_{-j-1}=c,\,
Y_{-k}^{-j-2}=w_{-k}^{-j-2}\big)\big]
$$
$$
\times \proba\big(Y_{-j-1}=c \tq X_{-j-1}^{-1}=w_{-j-1}^{-1},\,
Y_{-k}^{-j-2}=w_{-k}^{-j-2}\big)\;.
$$
The term with $c=w_{-j-1}$ in the above sum vanishes while for $c\neq
w_{-j-1}$
we can apply the first part of Lemma \ref{lescon2} to each term
in the square bracket and get
$$
\bigg|\sum_{c}
\big[\proba\big(X_{0}=a\tq X_{-j-1}^{-1}=w_{-j-1}^{-1}
Y_{-k}^{-j-1}=w_{-k}^{-j-1}\big)
$$
$$
-\proba\big(X_{0}=a\tq X_{-j-1}^{-1}=w_{-j-1}^{-1},\,Y_{-j-1}=c,\,
Y_{-k}^{-j-2}=w_{-k}^{-j-2}\big)\big]
$$
$$
\times  \proba\big(Y_{-j-1}=c \tq X_{-j-1}^{-1}=w_{-j-1}^{-1},\,
Y_{-k}^{-j-2}=w_{-k}^{-j-2}\big)\bigg|
$$
%\marginpar{c'est le meme p dans les 2 termes}
$$
\le 2\,\big(e^{\beta_{j+1,\,k}}-1\big)\;
\sum_{c\neq w_{-j-1}}\proba\big(Y_{-j-1}=c \tq X_{-j-1}^{-1}=w_{-j-1}^{-1},\,
Y_{-k}^{-j-2}=w_{-k}^{-j-2}\big)
$$
$$
=2\,\big(e^{\beta_{j+1,\,k}}-1\big)\;\proba\big(Y_{-j-1} \neq w_{-j-1}
\tq X_{-j-1}^{-1}=w_{-j-1}^{-1},\,
Y_{-k}^{-j-2}=w_{-k}^{-j-2}\big)
$$
$$
\le 2\;\rho\;e^{\somb_{j,\,k}}\;\big(e^{\beta_{j+1,\,k}}-1\big)
$$
by the second part of Lemma \ref{lescon4}.

Collecting all the previous  estimates we get
$$
\big|\proba\big(X_{0}=a\tq Y_{-k}^{-1}=w_{-k}^{-1}\big) -
\proba\big(X_{0}=a\tq X_{-k}^{-1}=w_{-k}^{-1}\big) \big|
$$
$$
\le 
\sum_{j=0}^{k-1}\left(
 \frac{\rho\;e^{2\,\somb_{j,\,k}}\;
\left[2\,\big(e^{\beta_{j+1,\,k}}-1\big)+\big(e^{2\,\beta_{j+1,\,k}}-1\big)\right]}{
\alpha\;(1-\rho)^{2}}
+2\;\rho\;e^{\somb_{j,\,k}}\;\big(e^{\beta_{j+1,\,k}}-1\big)
\right)
$$
$$
\le 
\left(
 \frac{e^{2\,\somb_{k,\,k}}\;
\left[2\,\big(e^{\somb_{k,\,k}}-1\big)+\big(e^{2\,\somb_{k,\,k}}-1\big)\right]}{
\alpha\;(1-\rho)^{2}}
+2\;e^{\somb_{k,\,k}}\;\big(e^{\somb_{k,\,k}}-1\big)
\right)\rho
$$
since from $\beta_{j,\,k}\ge 0$ we have
%\marginpar{preuve par recurence sur le nb de termes}
$$
\sum_{j=0}^{k} \left(e^{\beta_{j,\,k}}-1\right)\le e^{\somb_{k,\,k}}-1\;.
$$
%\marginpar{Verifier rearrangement final}
The first part of the theorem follows.

From the second part of Lemma \ref{nonj} and the first part of Theorem
we obtain
$$
\bigg|
\proba\big(Y_{0}=a\tq Y_{-k}^{-1}=w_{-k}^{-1}\big)-
\proba\big(X_{0}=a\tq X_{-k}^{-1}=w_{-k}^{-1}\big)\bigg|
$$
$$
\le\rho\;\frac{R(\alpha,\,j,\,\rho)}{\alpha}
\;\proba\big(X_{0}=a\tq X_{-k}^{-1}=w_{-k}^{-1}\big)\;,
$$
and the second part of the Theorem follows.
\end{proof}

\begin{proof}[Proof of Proposition \ref{leY}.]
We have
$$
\frac{\proba\big(Y_{0}=a\tq Y_{-j}^{-1}=y_{-j}^{-1},\,
Y_{-k}^{-j-1}=y_{-k}^{-j-1}\big)}
{\proba\big(Y_{0}=a\tq Y_{-j}^{-1}=y_{-j}^{-1},\,
Y_{-k}^{-j-1}=\tilde y_{-k}^{-j-1}\big)}
$$
$$
=
\frac{\proba\big(Y_{0}=a\tq Y_{-j}^{-1}=y_{-j}^{-1},\,
Y_{-k}^{-j-1}=y_{-k}^{-j-1}\big)}
{\proba\big(X_{0}=a\tq X_{-j}^{-1}=y_{-j}^{-1},\,
X_{-k}^{-j-1}=y_{-k}^{-j-1}\big)}\;\times
$$
$$
\frac{\proba\big(X_{0}=a\tq X_{-j}^{-1}=y_{-j}^{-1},\,
X_{-k}^{-j-1}=\tilde y_{-k}^{-j-1}\big)}
{\proba\big(Y_{0}=a\tq Y_{-j}^{-1}=y_{-j}^{-1},\,
Y_{-k}^{-j-1}=\tilde y_{-k}^{-j-1}\big)}\;\times
$$
$$
\frac{\proba\big(X_{0}=a\tq X_{-j}^{-1}=y_{-j}^{-1},\,
X_{-k}^{-j-1}=y_{-k}^{-j-1}\big)}
{\proba\big(X_{0}=a\tq X_{-j}^{-1}=y_{-j}^{-1},\,
X_{-k}^{-j-1}=\tilde y_{-k}^{-j-1}\big)}\;,
$$
and the first result follows using twice the second part of Proposition
\ref{leX} 
and twice Theorem \ref{blurred}. 

The second result follows at once from the first one and the identity
%\marginpar{par B3}
$$
\proba\big(Y_{0}=a\tq Y_{-j}^{-1}=y_{-j}^{-1}\big)
$$
$$
=\!\!\!\sum_{y_{-k}^{-j-1}}\!
\proba\big(Y_{-k}^{-j-1}=y_{-k}^{-j-1}\tq
Y_{-j}^{-1}=y_{-j}^{-1}\big)\,
\proba\big(Y_{0}=a\tq Y_{-k}^{-j-1}=y_{-k}^{-j-1},\,
Y_{-j}^{-1}=y_{-j}^{-1}\big).
$$
The third result follows using Theorem \ref{blurred} and the second
part of Proposition \ref{leX}.
\end{proof}

\begin{proof}[Proof of Theorem \ref{predire}]
The first part is the result in Lemma \ref{undeplus}. 

For the second part we have using Lemma \ref{undeplus} and the third part
of Theorem \ref{leY}
$$
\proba\big(X_{0}=a\tq Y_{-k}^{-1}=w_{-k}^{-1}\big)
\le \proba\big(Y_{0}=a\tq Y_{-k}^{-1}= w_{-k}^{-1}\big)+ \rho
$$
$$
\le \proba\big(Y_{0}=a\tq Y_{-k}^{-1}= w_{-k}^{-1}\big)\;\left(1+
\frac{\rho}{\alpha-\rho\;R(\alpha,\,k,\,\rho)}\right)\;.
$$
The lower bound follows similarly.
\end{proof}

%\bibitem{}

\end{document}